\documentclass[11pt]{article}%
\usepackage[textwidth=6.0in,textheight=9.0in,centering]{geometry}
\usepackage{stix2}
\usepackage{amsxtra,amscd}
\usepackage{tikz-cd}
\usepackage{graphicx}
\usepackage[amsmath,hyperref,thmmarks]{ntheorem}
\usepackage{natbib}
\usepackage{verbatim}
\usepackage{makeidx}
\usepackage{amsmath}
\usepackage{amsfonts}
\usepackage{amssymb}
\theoremnumbering{arabic}
\theoremheaderfont{\scshape}
\RequirePackage{latexsym}
\theorembodyfont{\slshape}
\theoremseparator{}
\newtheorem*{theorem}{Theorem}
\theorembodyfont{\upshape}
\newtheorem*{problem}{Problem}
\theorembodyfont{\normalsize}

\theoremstyle{nonumberplain}
\theoremheaderfont{\sc}
\theorembodyfont{\normalfont}
\theoremsymbol{\ensuremath{_\Box}}
\RequirePackage{amssymb}
\newtheorem{proof}{Proof.}
\qedsymbol{\ensuremath{_\Box}}
\theoremclass{LaTeX}
\makeindex
\newcommand{\df}{\smash{\lower.12em\hbox{\textup{\tiny def}}}}
\newcommand{\lsimeq}{\smash{\lower.24em\hbox{$\scriptstyle\simeq$}}}
\DeclareMathOperator{\Gal}{Gal}
\DeclareMathOperator{\Spec}{Spec}
\tikzcdset{arrow style=math font}
\begin{document}

\title{On the integral Tate conjecture for abelian varieties}
\date{September 8, 2025}
\author{J.S. Milne}
\maketitle

\begin{abstract}
Recently Engel et al.\ (2025) have shown that the integral Hodge conjecture fails for
 very general abelian varieties. Using Deligne's theory of absolute Hodge
cycles, we deduce a similar statement for the integral Tate conjecture.

\end{abstract}

For a smooth projective variety $X$ over a subfield $k$ of $\mathbb{C}{}$, we
let $C^{r}(X)$ denote the $\mathbb{Z}$-module of algebraic cycles of
codimension $r$ on $X$ modulo (Betti) homological equivalence. When $k$ is
finitely generated over $\mathbb{Q}$, the integral Tate conjecture\footnote{Not
stated by Tate. At the time Tate stated his conjectures, Atiyah and Hirzebruch
had already found their counterexample to the integral Hodge conjecture.} says
that the cycle class map%
\[
C^{r}(X)\otimes\mathbb{Z}_{\ell}
\xrightarrow{c^{\prime}}H_{\mathrm{et}}^{2r}(\bar
{X},\mathbb{Z}{}_{\ell}(r))^{\Gal(\bar{k}/k)},
\]
is surjective. Here $\bar{k}$ is the algebraic closure of $k$ in $\mathbb{C}%
{}$ and $\bar{X}=X\times\Spec\bar{k}$.

\begin{theorem}
Let $A_{0}$ be an abelian variety over a subfield $k$ of $\mathbb{C}{}$
finitely generated over $\mathbb{Q}{}$. If the integral Hodge conjecture fails
for $A\overset{\df}{=}(A_{0})_{\mathbb{C}{}}$, then the integral Tate
conjecture fails for $A_{0}$ over some finite extension of $k$.
\end{theorem}

\begin{proof}
We use that the cohomology of abelian varieties has no torsion to identify the
integral cohomology groups with subgroups of the rational groups.

Let $\gamma\in H^{2r}(A,\mathbb{Z}(r))\cap H^{0,0}$ be an integral Hodge class
on $A$ that is not algebraic, i.e., not in the image of the cycle class map%
\[
C^{r}(X)\xrightarrow{c}H^{2r}(A,\mathbb{Z}{}(r)).
\]
For some $\ell$, $\gamma\otimes1$ is not in the image of%
\begin{equation}
C^{r}(X)\otimes\mathbb{Z}{}_{\ell}\xrightarrow{c\otimes 1}H^{2r}%
(A,\mathbb{Z}{}(r))\otimes\mathbb{Z}{}_{\ell}.\tag{*}%
\end{equation}
According to the main theorem of Deligne 1982, the image of $\gamma$ in
$H_{\mathbb{A}{}}^{2r}(A)(r)$ is an absolute Hodge class on $A$. According
Proposition 2.9, ibid., $\gamma$ arises from an absolute Hodge class
$\gamma_{0}$ on $\bar{A}_{0}$ and $\mathrm{Gal}(\bar{k}/k)$ acts on
$\gamma_{0}$ through a finite quotient. After passing to a finite extension of
$k$, we may suppose that $\Gal(\bar{k}/k)$ fixes $\gamma_{0}$.  Then the
$\ell$-component $\gamma_{0}(\ell)$ of $\gamma_{0}$ is a Tate class in
$H^{2r}(\bar{A}_{0},\mathbb{Q}_{\ell}(r))$. Consider the commutative diagram%
\[
\begin{tikzcd}
C^{r}(A)\otimes\mathbb{Z}_{\ell}\arrow{r}{c\otimes 1}
&H^{2r}(A,\mathbb{Z}(r))\otimes
\mathbb{Q}_{\ell}\\
C^{r}(A_0)\otimes\mathbb{Z}_{\ell}\arrow{r}{c^{\prime}}\arrow{u}
&H_{\mathrm{et}}^{2r}(\bar{A}_0,\mathbb{Q}_{\ell}(r)),\arrow{u}[swap]{\simeq}
\end{tikzcd}
\]
where the isomorphism at right is the composite of the canonical isomorphisms%
\[
H_{\mathrm{et}}^{2r}(\bar{A}_{0},\mathbb{Q}{}_{\ell}(r))\simeq 
H_{\mathrm{et}}^{2r}(A,\mathbb{Q}{}_{\ell
}(r))\simeq H^{2r}(A,\mathbb{Z}{}(r))\otimes\mathbb{Q}{}_{\ell}.
\]
Under this isomorphism $\gamma_{0}(\ell)$ corresponds to $\gamma\otimes1$, so
$\gamma_{0}(\ell)$ lies in $H^{2r}(\bar{A}_{0},\mathbb{Z}{}_{\ell}(r))$ and is
an integral Tate class on $A_{0}$. If $\gamma_{0}(\ell)$ is in the image of
$c^{\prime}$, then $\gamma\otimes1$ is in the image of $c\otimes1$,
contradicting (*). We have shown that the integral Tate conjecture fails for
$A_{0}$.
\end{proof}

Engel et al.\ 2025 show that very general principally polarized abelian
varieties over $\mathbb{C}$ of dimension $\geq 4$ fail the integral Hodge conjecture. Thus, a
principally polarized abelian variety over a finitely generated 
subfield $k$ of
$\mathbb{C}$ fails the integral Tate conjecture over some  
finite extension of $k$ if, over $\mathbb{C}$, it becomes \textquotedblleft very
general\textquotedblright\ in their sense.

\begin{problem}
Exhibit an abelian variety over a number field that fails the integral Tate conjecture.
\end{problem}

\noindent\textbf{References}

\begingroup
\parindent 0pt \everypar{\hangindent1.5em\hangafter1}
\medskip Deligne, P. 1982. Hodge
cycles on abelian varieties (notes by J.S. Milne), pp. 9--100. In Hodge cycles,
motives, and Shimura varieties, Lecture Notes in Mathematics. Springer-Verlag,
Berlin. 

Engel, P., de Gaay Fortman, O., Schreieder, S., 2025. 
Matroids and the integral Hodge conjecture for abelian varieties,
arXiv:2507.15704v2.

\endgroup

\end{document}